\documentclass[a4paper,12pt]{amsart}
\usepackage[utf8]{inputenc}
\usepackage{csquotes}
\usepackage{amsmath, amssymb, amsthm, xcolor, enumerate}
\usepackage[english]{babel}
\usepackage{tikz-cd}
\usepackage[all]{xy}
\usepackage[backref]{hyperref}
\usepackage{thmtools}
\usepackage{mathtools}
\usepackage{microtype}

\newtheorem{theorem}{Theorem}[section]

\newtheorem{prop}[theorem]{Proposition}
\newtheorem{lemma}[theorem]{Lemma}

\theoremstyle{definition}

\newtheorem{defn}[theorem]{Definition}
\newtheorem{question}{Question}

\title{A note on when amenable traces are quasidiagonal}

\author{Robert Neagu}
\address{\hskip-\parindent Robert Neagu, Mathematical Institute, University of Oxford, Oxford, OX2 6GG, UK.}
\email{robert.neagu@maths.ox.ac.uk}

\thanks{The author was supported by the EPSRC grant EP/R513295/1.}

\begin{document}
\maketitle

\begin{abstract}
    We will show that for a separable exact $C^*$-algebra with a faithful amenable trace, the property that all amenable traces are quasidiagonal is invariant under homotopy.
\end{abstract}

\newcommand{\nuc}{\operatorname{nuc}}
\newcommand{\ev}{\operatorname{ev}}
\newcommand{\id}{\operatorname{id}}
\newcommand{\Ext}{\operatorname{Ext}}
\newcommand{\tr}{\operatorname{tr}}
\newcommand{\II}{\operatorname{II}}
\newcommand{\Ad}{\operatorname{Ad}}
\newcommand{\UCT}{\operatorname{UCT}}
\newcommand{\UHF}{\operatorname{UHF}}
\newcommand{\VII}{\operatorname{VII}}
\newcommand{\KK}{\operatorname{KK}}

\section*{Introduction}
\renewcommand*{\thetheorem}{\Alph{theorem}}

Over the years, many $C^*$-properties have been proved to have a strong topological flavour. One of the first such results was obtained by Voiculescu in \cite{voiculescu}, where he proved that quasidiagonality, an external finite dimensional approximation property, is invariant under homotopy. Then, Dadarlat defined in \cite{dadarlat} quasidiagonality for asymptotic morphisms and proved that this property is still invariant under an appropriate notion of homotopy. 

Developing such approximation properties to traces proved very fruitful. Amenability of traces dates back to the work of Connes \cite{connes}, where he essentially showed that a $\II_1$-factor is injective if and only if its unique normalised trace is amenable. Later, quasidiagonal traces were introduced by Brown in \cite{qdtraces} and they proved to be a fundamental tool in obtaining the classification theorem of simple nuclear $C^*$-algebras, via the quasidiagonality theorem of Tikuisis, White, and Winter \cite{TWW}. In \cite{qdtraces}, Brown also observed that any quasidiagonal trace is amenable and asked if the converse also holds. The converse can be viewed as a generalisation of Rosenberg's Conjecture, asking if any group $C^*$-algebra of a discrete amenable group is quasidiagonal. This was answered affirmatively in \cite{TWW}. Until the quasidiagonality theorem \cite{TWW}, not so many results were known in this direction. One such result shows that if $A$ is any $C^*$-algebra, then any amenable trace on its cone $C_0((0,1],A)$ is, in fact, quasidiagonal {\cite[Proposition 3.2]{BCW}}. 

\begin{question}\label{brownq}
Are all amenable traces quasidiagonal?
\end{question}

Schafhauser introduced in \cite{chris1} new techniques for proving the quasidiagonality theorem of Tikuisis, White, and Winter \cite{TWW}. We will make use of these techniques to investigate the behaviour of Question \ref{brownq} under homotopy. The main results we are going to prove are the following.

\begin{theorem}\label{no1}
Let $A$ be a separable exact $C^*$-algebra with a faithful amenable trace $\tau$. If $\sigma:A\rightarrow A$ is a $^*$-homomorphism which is homotopic to the identity map on $A$ and $\tau\circ\sigma$ is a quasidiagonal trace on $A$, then $\tau$ is quasidiagonal on $A$.
\end{theorem}

\begin{theorem}\label{no2}
Let $A$ be a separable exact $C^*$-algebra with a faithful amenable trace and suppose $A$ is homotopy dominated by some $C^*$-algebra $B$. If all amenable traces on $B$ are quasidiagonal, then all amenable traces on $A$ are quasidiagonal.
\end{theorem}

\subsection*{Acknowledgements} The author would like to thank Stuart White and Chris Schafhauser for helpful comments.

\allowdisplaybreaks

\section{Preliminaries and the contractible case}
Here, by a \emph{trace} $\tau$ on a $C^*$-algebra $A$ we mean a positive \emph{contractive} linear functional such that $\tau(ab)=\tau(ba)$ for all $a,b\in A$. In particular, if $A$ is unital and $\tau(1)=1$, then we say $\tau$ is a \emph{tracial state}. Since we are going to consider traces composed with $^*$-homomorphisms, we cannot expect that traces of this form will be states, so we need to consider all tracial functionals with norm less than or equal to $1$.

\begin{defn}\label{amqdef}
Let $A$ be a separable $C^*$-algebra.\begin{enumerate}
\item  A trace $\tau:A\rightarrow\mathbb{C}$ is called \emph{amenable} if for all $n\in\mathbb{N}$, there is an integer $k(n)\geq 1$ and a cpc map $\phi_n:A\rightarrow \mathbb{M}_{k(n)}(\mathbb{C})$ such that $$\|\phi_n(ab)-\phi_n(a)\phi_n(b)\|_2 \rightarrow 0$$ and $\tr_{k(n)}(\phi_n(a))\rightarrow \tau(a)$ for all $a,b\in A$, where $\|x\|_2=\tr_{k(n)}(x^*x)^{1/2}$ and $\tr_{k(n)}$ is the unique normalised trace on $\mathbb{M}_{k(n)}(\mathbb{C})$.
\item A trace $\tau:A\to\mathbb{C}$ is \emph{quasidiagonal} if for all $n\in \mathbb{N}$ there is an integer $k(n)\geq 1$ and a cpc map $\phi_n:A\rightarrow \mathbb{M}_{k(n)}(\mathbb{C})$ such that $$\|\phi_n(ab)-\phi_n(a)\phi_n(b)\|\rightarrow 0$$ and $\tr_{k(n)}(\phi_n(a))\rightarrow \tau(a)$ for all $a,b\in A$.
    \end{enumerate}
\end{defn}

Throughout, $\omega$ will stand for a fixed free ultrafilter on $\mathbb{N}$. Then, we can define $$\mathcal{Q}_{\omega} := \ell^{\infty}(\mathcal{Q})/\{(x_n)_{n=1}^{\infty} \in \ell^{\infty}(\mathcal{Q}) : \lim_{n\to\omega} \|x_n\|= 0\}$$ to be the uniform ultrapower of the universal $\UHF$-algebra $\mathcal{Q}$. Similarly, 

$$\mathcal{R}^{\omega} := \ell^{\infty}(\mathcal{R})/\{(x_n)_{n=1}^{\infty} \in \ell^{\infty}(\mathcal{R}) : \lim_{n\to\omega} \|a_n\|_2= 0\}$$ represents the tracial ultrapower of the hyperfinite $\II_1$-factor $\mathcal{R}$. Let $\tr_{\omega}$ and $\tr^{\omega}$ denote the
traces on $\mathcal{Q}_{\omega}$ and $\mathcal{R}^{\omega}$ induced by the unique traces on $\mathcal{Q}$ and $\mathcal{R}$. 

A standard application of Kaplansky's density theorem shows that the map $\pi:\mathcal{Q}_{\omega}\to\mathcal{R}^{\omega}$ induced by the canonical inclusion $\mathcal{Q}\hookrightarrow\mathcal{R}$ is a surjection. Therefore, one obtains an extension
\[
\begin{tikzcd}
0\ar{r} & J \ar{r} & \mathcal{Q}_{\omega} \ar{r}{\pi} & \mathcal{R}^{\omega} \ar{r} & 0,
\end{tikzcd}
\]known as the trace-kernel extension. Schafhauser's breakthrough rephrased the quasidiagonality of a trace into a lifting problem. For convenience, we will record the following result which is well-known to experts, (see {\cite[Proposition 1.3]{chris1}}).

\begin{prop}\label{keyprop}
Let $A$ be a separable $C^*$-algebra.
\begin{enumerate}
\item A trace $\tau$ is amenable on $A$ if and only if there is a $^*$-homomorphism $\phi:A\rightarrow \mathcal{R}^{\omega}$ with a cpc lift $A\rightarrow \ell^{\infty}(\mathcal{R})$ such that $\tr^{\omega}\circ\phi=\tau$.

\item A trace $\tau$ is quasidiagonal on $A$ if and only if there is a $^*$-homomorphism $\phi:A\rightarrow \mathcal{Q}_{\omega}$ with a cpc lift $A\rightarrow \ell^{\infty}(\mathcal{Q})$ such that $\tr_{\omega}\circ\phi=\tau$.
\end{enumerate}
\end{prop}

We are now in the position to state the key idea in Schafhauser's approach in \cite{chris1}. If $\tau$ is any amenable trace on a separable exact $C^*$-algebra $A$, Proposition \ref{keyprop} gives a trace-preserving $^*$-homomorphism $\phi:A\to\mathcal{R}^{\omega}$ which has a cpc lift into $\ell^{\infty}(\mathcal{R})$. Putting together $\phi$ and the canonical quotient map $\pi$, we get the following pullback extension
\begin{equation}\label{pullbackext}
\begin{tikzcd}
\eta_0 :0\ar{r}& J\ar{r}\ar{d}{\id} & E_0 \ar{r}\ar{d} & A\ar{r}\ar{d}{\phi} & 0\\
\eta   :0\ar{r} & J \ar{r} & \mathcal{Q}_{\omega} \ar{r}{\pi} & \mathcal{R}^{\omega} \ar{r} & 0.
\end{tikzcd}
\end{equation} Recall that a map $\sigma:A\to E_0$ is \emph{weakly nuclear} if for all $x\in J$, the map $A\to J$ given by $a \mapsto x\sigma(a)x^*$ is nuclear. Since $\eta_0$ is a pullback extension, one can note that $\eta_0$ has a $^*$-homomorphic splitting $\sigma:A\to E_0$ if and only if there exists a lift $\psi:A\to\mathcal{Q}_{\omega}.$ The only if direction is clear, and if we have a lift, then we can define $\sigma(a)=a\oplus\psi(a)$. Note that if $\psi$ is nuclear, then $\sigma$ is weakly nuclear.

The key point is that $\eta_0$ has a weakly nuclear $^*$-homomorphic splitting if and only if $\tau$ is quasidiagonal. The only if direction appears in {\cite[Theorem 1.2]{chris1}}, and the converse, even not spelt out explicitly, is contained in the proof of {\cite[Proposition 4.3]{chris2}}. Let us include a proof for the convenience of the reader.

\begin{lemma}\label{qdsplits}
Let $\tau$ be an amenable trace on a separable exact $C^*$-algebra $A$. Then the extension $\eta_0$ constructed above has a weakly nuclear $^*$-homomorphic splitting if and only if $\tau$ is quasidiagonal. 
\end{lemma}

\begin{proof}
The only if direction is shown in {\cite[Theorem 1.2]{chris1}}. Conversely, suppose $\tau$ is quasidiagonal. Proposition \ref{keyprop} then implies that there exists a trace-preserving $^*$-homomorphism $\psi:A\rightarrow\mathcal{Q}_{\omega}$ with a cpc lift $A\rightarrow \ell^{\infty}(\mathcal{Q})$. Since $A$ is exact, {\cite[Proposition 3.1]{gabe2}} gives that $\psi$ is nuclear.

Then, we have that $\tr^{\omega}\circ\phi=\tr^{\omega}\circ(\pi\circ\psi)$, with $\pi\circ\psi$ and $\phi$ nuclear by exactness of $A$ {\cite[Lemma 5.1]{chris1}}. Since $\mathcal{R}^{\omega}$ is a finite factor, by a consequence of Connes' theorem \cite{connes}, these two maps are approximmately unitarily equivalent (see {\cite[Proposition 1.1]{chris2}}). Moreover, since $A$ is separable, by a reindexing argument {\cite[Lemma 4.1]{reindexing}}, $\phi$ and $\pi\circ\psi$ are unitarily equivalent. Let $\tilde{v}$ be a unitary in $\mathcal{R}^{\omega}$ such that $\phi\circ\sigma=\Ad(\tilde{v})\pi\circ\psi$. As the unitary group of $\mathcal{R}^{\omega}$ is path-connected, there exists a unitary $v\in \mathcal{Q}_{\omega}$ such that $\pi(v)=\tilde{v}$. Therefore, replacing $\psi$ by $\Ad(v)\psi$, we can assume that $\pi\circ\psi=\phi\circ\sigma$, i.e. the extension $\eta_0$ splits. But $\psi$ is nuclear, so $\eta_0$ has a weakly nuclear $^*$-homomorphic splitting.
\end{proof}

Before proving our main results, we take a short detour and examine the case when $A$ is a contractible $C^*$-algebra. Since contractible $C^*$-algebras are homotopy equivalent to $0$, this is an instance where the property that all amenable traces are quasidiagonal is homotopy invariant.

\begin{prop}\label{contr}
Let $A$ be a separable contractible $C^*$-algebra. Then all amenable traces on $A$ are quasidiagonal.
\end{prop}

\begin{proof}
Let $\tau$ be an amenable trace on $A$. Since $A$ is contractible, the identity map on $A$ is homotopic to the zero map. Therefore, there exists a $^*$-homomorphism $\theta:A\rightarrow C_0((0,1],A)$ such that $\theta(a)(1)=a$ for all $a\in A$.

Now observe that $\tau$ factorises as

\[
\begin{tikzcd}
    A \ar{r}{\theta} & C_0((0,1],A) \ar{r}{\tau\circ\ev_1} & \mathbb{C}.
\end{tikzcd}
\] Since $\tau$ is amenable on $A$ and $\ev_1$ is a $^*$-homomorphism, $\tau\circ\ev_1$ is amenable on the cone $C_0((0,1],A)$, so quasidiagonal {\cite[Proposition 3.2]{BCW}}. Therefore, there exists a $^*$-homomorphism $\phi:C_0((0,1],A)\rightarrow \mathcal{Q}_{\omega}$ with a cpc lift $\psi:C_0((0,1],A)\rightarrow \ell^{\infty}(\mathcal{Q})$ such that $\tr_{\omega}\circ\phi=\tau\circ\ev_1$, where $\tr_{\omega}$ is the induced trace on $\mathcal{Q}_{\omega}$.

Now $\phi\circ \theta: A\rightarrow \mathcal{Q}_{\omega}$ is a $^*$-homomorphism with a cpc lift $\psi\circ\theta:A\rightarrow \ell^{\infty}(\mathcal{Q})$ such that $$\tr_{\omega}\circ\phi\circ\theta= (\tau\circ\ev_1)\circ\theta=\tau.$$ Thus, $\tau$ is a quasidiagonal trace on $A$.
\end{proof}

Let us end this section with an observation that we are going to use in the proof of Theorem \ref{no1}. This is in the spirit of deunitization tricks used in \cite{TWW} and \cite{chris1}. Precisely, the first part is very similar to the techniques used in {\cite[Theorem 1.2]{chris1}}, and the latter part is {\cite[Proposition 1.4]{TWW}}, without assuming nuclearity of $A.$

\begin{lemma}\label{tracecontraction}
Let $\tau$ be an amenable trace on a separable $C^*$-algebra $A$. Then $\frac{1}{2}\tau$ is amenable. Moreover, if $\frac{1}{2}\tau$ is quasidiagonal, then $\tau$ is quasidiagonal. 
\end{lemma}

\begin{proof}
With the notation as in Definition \ref{amqdef}, let $\phi_n:A\to\mathbb{M}_{k(n)}(\mathbb{C})$ be cpc maps, approximately multiplicative in $2$-norm approximating the trace $\tau$. If we denote by $\iota_n$ the canonical embedding into the top left corner $\mathbb{M}_{k(n)}(\mathbb{C})\to \mathbb{M}_{2k(n)}(\mathbb{C})$, let $\psi_n=\iota_n\circ\phi_n$. Then, $\psi_n$ is a sequence of cpc maps, approximately multiplicative in $2$-norm such that $\tr_{2k(n)}(\psi_n(a))\to \frac{1}{2}\tau(a)$ for all $a\in A$. Thus, the trace $\frac{1}{2}\tau$ is amenable.

For the last part, if $\frac{1}{2}\tau$ is quasidiagonal, then we can follow the strategy in the implication $(ii)(c) \implies (ii)(b)$ in {\cite[Proposition 1.4]{TWW}}. Suppose there exist approximately multiplicative cpc maps $\psi_n:A\to\mathbb{M}_{k(n)}(\mathbb{C})$ approximating the trace $\frac{1}{2}\tau$. If $A$ is unital, then, for $n$ large enough, $\psi_n(1)$ is approximately a projection, so let $p_n\in\mathbb{M}_{k(n)}(\mathbb{C})$ be a projection close to $\psi_n(1)$. Then, we can consider $\psi_n$ as a cpc map into the corner $p_n\mathbb{M}_{k(n)}(\mathbb{C})p_n$ and it is still approximately multiplicative since $\psi_n(1)$ approximately commutes with the image of $\psi_n$. Finally, exactly as in {\cite[Proposition 1.4]{TWW}}, one can note that the sequence of cpc maps $(\psi_n)_{n\geq 1}$ approximates the trace $\tau$. Thus, $\tau$ is quasidiagonal.

If $A$ is non-unital, then we can pass to the unitization and the same proof follows since a trace on $A$ is quasidiagonal if and only if the induced trace on the unitization is quasidiagonal{\cite[Proposition 3.5.10]{qdtraces}}.
\end{proof}

\section{Main results}

The proof of Theorem \ref{no1} is heavily motivated by Theorem $1.2$ in \cite{chris1}. Essentially, we can build two extensions with homotopic Busby invariants and, since one splits, the other will split as well. We refer the reader to {\cite[Section 2]{chris1}} or {\cite[Chapter $\VII$]{blackadar}} for an introduction to the theory of extensions. There are \emph{three} key steps in the proof of {\cite[Theorem 1.2]{chris1}}: obtain a separable version of the extension $\eta_0$ built in \eqref{pullbackext}, show it is nuclearly absorbing, and that it has class $0$ in $\Ext_{\nuc}$. Here, an extension has class $0$ if it has a weakly nuclear splitting after taking the direct sum with an extension with a weakly nuclear splitting. Moreover, an extension is called \emph{nuclearly absorbing} if taking the direct sum with any extension with a weakly nuclear $^*$-homomorphic splitting we obtain the same extension up to an appropriate notion of equivalence.

In {\cite[Theorem 1.2]{chris1}}, the $\UCT$ is needed to show that the relevant $\Ext_{\nuc}$ class vanishes. We are going to avoid assuming the $\UCT$ by using that the relevant class in $\Ext_{\nuc}$ is preserved under homotopy of the Busby invariant.

Therefore, we will break the proof into propositions illustrating these steps. Let us start by producing a separable version of \eqref{pullbackext}. Crucially, we will use the fact that the property of being an admissible kernel (in the sense of {\cite[Definition 2.1]{chris1}}) is separably inheritable {\cite[Proposition 4.1]{chris1}}. The following result is essentially Proposition $4.2$ of \cite{chris1}, with the only modification being that we can make sure that $B_0$ contains a specified separable preimage of $\pi$. Note that the case $C=\{0\}$ is exactly Proposition $4.2$ in \cite{chris1}. 

\begin{prop}[cf. {\cite[Proposition 4.2]{chris1}}]\label{separable}
Consider an extension 
\[
\begin{tikzcd}
0 \ar{r} & I \ar{r} & B \ar{r}{\pi} & D \ar{r} & 0
\end{tikzcd}
\]
such that $I$ is an admissible kernel, $B$ and $D$ are unital. Suppose $D_0\subset D$ is any separable, unital subalgebra and $C$ is a separable $C^*$-subalgebra of $B$ such that $\pi(C)\subset D_0$. Then, there exists a separable, unital subalgebra $B_0\subset B$ such that $C\subset B_0$, $\pi(B_0)=D_0$ and $B_0\cap I$ is an admissible kernel. 
\end{prop}

\begin{proof}
Let $S_0$ be countable dense in $C$. Fix a countable dense subset $T\subset D_0$ such that $\pi(S_0)\subset T$ and let $S\subset B$ countable such that $S_0\subset S$ and $\pi(S)=T$. The rest of the proof now follows verbatim as in {\cite[Proposition 4.2]{chris1}}.
\end{proof}

Combining this and Lemma \ref{qdsplits}, one can show that a quasidiagonal trace produces a separable extension with a weakly nuclear $^*$-homomorphic splitting. To show this, let us move into the set-up of Theorem \ref{no1}, but without assuming faithfulness. Let $A$ be a separable exact $C^*$-algebra, $\tau$ an amenable trace on $A$, and $\sigma:A\rightarrow A$ a $^*$-homomorphism which is homotopic to the identity map on $A$ such that $\tau\circ\sigma$ is a quasidiagonal trace on $A$. Using the notation in Proposition \ref{keyprop}, let $\phi:A\to\mathcal{R}^{\omega}$ be a $^*$-homomorphism witnessing amenability of $\tau$.

\begin{prop}\label{sepextsplits}
There exists an extension 
\[\begin{tikzcd}\label{separabletracekernel}
\eta' : 0\ar{r} & J_0\ar{r}  & Q_0\ar{r}{\pi_0} & R_0\ar{r} & 0
\end{tikzcd}
\] with $J_0$ a separable admissible kernel and $\phi(A)\subseteq R_0$. If we denote by $\phi_0$ the corestriction of $\phi$ to $R_0$, then the pullback extension induced by $\pi_0$ and $\phi_0\circ\sigma$ has a weakly nuclear $^*$-homomorphic splitting.
\end{prop}

\begin{proof}
Following Proposition $4.3$ in \cite{chris1}, we get a separable, unital $C^*$-subalgebra $R_0\subset\mathcal{R}^{\omega}$ such that $R_0$ is simple, $\phi(A)\subset R_0$ and $\phi$ is nuclear as a map $A\rightarrow R_0$. Denote by $\phi_0$ the corestriction of $\phi$ to $R_0$. Applying Proposition \ref{separable} to the trace-kernel extension $\eta$, we obtain a separable, unital subalgebra $Q_0\subset\mathcal{Q}_{\omega}$ such that $J_0=Q_0\cap J$ is an admissible kernel and $\pi(Q_0)=R_0$. We denote by $\pi_0$ the restriction of $\pi$ from $Q_0$ into $R_0$. Then, we can form a pullback extension.
\begin{equation}\label{sepsplitext}
\begin{tikzcd}
\eta_1' : 0\ar{r}  & J_0\ar{r}\ar{d}{\id} & E_1'\ar{r}\ar{d} & A\ar{r}\ar{d}{\phi_0\circ\sigma} & 0\\
\eta' : 0\ar{r} & J_0\ar{r}  & Q_0\ar{r}{\pi_0} & R_0\ar{r} & 0
\end{tikzcd}
\end{equation}

By Lemma \ref{qdsplits}, the pullback extension induced by $\pi$ and $\phi\circ\sigma$ has a weakly nuclear $^*$-homomorphic splitting induced by $\psi:A\to\mathcal{Q}_{\omega}$. Since $\pi(\psi(A))=\phi\circ\sigma(A)\subset R_0$, Proposition \ref{separable} allows us to assume that $\psi(A)\subset Q_0$, so $\psi$ will also give a weakly nuclear $^*$-homomorphic splitting for the extension $\eta_1'$.
\end{proof}

With the notation from Proposition \ref{sepextsplits}, we consider the pullback extension \begin{equation}\label{septargetext}
\begin{tikzcd}
\eta_0' :0\ar{r}& J_0\ar{r}\ar{d}{\id} & E_0' \ar{r}\ar{d} & A\ar{r}\ar{d}{\phi_0} & 0\\
\eta'  :0\ar{r} & J_0 \ar{r} & Q_0 \ar{r}{\pi_0} & R_0 \ar{r} & 0
\end{tikzcd}
\end{equation} induced by $\pi_0$ and $\phi_0$, and we will prove that $\eta_0'$ has a weakly nuclear $^*$-homomorphic splitting.

If $A$ and $B$ are separable $C^*$-algebras, a direct consequence of the canonical identification $\KK^1(A,B)\cong \Ext^{-1}(A,B)$ {\cite[Corollary 18.5.4]{blackadar}}, says that two semisplit extensions with homotopic Busby invariants have the same class in $\Ext^{-1}(A,B)$ {\cite[Corollary 15.10.1]{blackadar}}. The key technical fact we are claiming is that the same holds when all extensions are weakly nuclear. Precisely, if $\beta_1$ and $\beta_2$ are nuclear Busby invariants, homotopic via a path of nuclear maps, then the extensions they induce have the same class in $\Ext_{\nuc}$. This is a direct consequence of
{\cite[Corollary 1.8]{kucerovsky}}. In {\cite[Corollary 1.8]{kucerovsky}}, Kucerovsky and Ng assume that all extensions are weakly nuclear, and by an absorbing extension they mean an extension which absorbs a weakly nuclear split extension. With these clarifications, {\cite[Corollary 1.8]{kucerovsky}} translates to $\KK_{\nuc}^1(A,B)\cong \Ext_{\nuc}(A,B)$. Therefore, if we take $\beta_1$ and $\beta_2$ as above, since the homotopy is induced by nuclear maps, Theorem $4.4$ in \cite{chris1} shows that the extensions induced by $\beta_1,\beta_2$, and the maps realising the homotopy are weakly nuclear. Since $\KK_{\nuc}^1(A,B)$ is invariant under homotopy via nuclear maps, $\beta_1$ and $\beta_2$ induce the same element in $\KK_{\nuc}^1(A,B)$. Thus, they induce extensions with the same class in $\Ext_{\nuc}(A,B)$ by {\cite[Corollary 1.8]{kucerovsky}}. Let us apply this observation to the extensions $\eta_0'$ and $\eta_1'$ defined in \eqref{septargetext} and \eqref{sepsplitext}.

\begin{prop}\label{class0}
With the notation from \eqref{septargetext}, the class of the extension $\eta_0'$ is $0$ in $\Ext_{\nuc}(A,J_0)$.
\end{prop}

\begin{proof}
Denote by $M(J_0)$ the multiplier algebra of $J_0$. Then, if $\beta:R_0\rightarrow M(J_0)/J_0=Q(J_0)$ is the Busby invariant of the extension $\eta'$, then $\beta\circ\phi_0$ is the Busby invariant of $\eta_0'$ and $\beta\circ\phi_0\circ\sigma$ is the Busby invariant of $\eta_1'$. Since $\sigma$ is homotopic to the identity on $A$, the Busby invariants of the extensions $\eta_0'$ and $\eta_1'$ are homotopic and since $\phi_0$ is nuclear, the homotopy is realised via nuclear maps. Therefore, the observation above implies that the extensions $\eta_0'$ and $\eta_1'$ have the same class in $\Ext_{\nuc}(A,J_0)$. 

Finally, $\eta_1'$ is a split extension by Proposition \ref{sepextsplits}, so $\eta_0'$ has class $0$ in $\Ext_{\nuc}(A,J_0)$. 
\end{proof}

The last step to conclude that $\eta_0'$ is a split extension is to show that $\eta_0'$ is nuclearly absorbing. This is the point where the faithfulness of $\tau$ comes into picture and, as in \cite{chris1}, it is required to prove that $\eta_0'$ is unitizably full (see {\cite[Section 2]{chris1}}). Since a unital $^*$-homomorphism cannot be unitizably full, we need to split into cases, and we first handle the case where $\phi_0$ is non-unital.

\begin{prop}\label{absorbing}
Let $A$ be a separable exact $C^*$-algebra, $\tau$ a faithful amenable trace on $A$, and $\sigma:A\rightarrow A$ a $^*$-homomorphism which is homotopic to the identity map on $A$ such that $\tau\circ\sigma$ is a quasidiagonal trace on $A$. Suppose further that $A$ is non-unital or $A$ is unital and $\phi_0(1)\neq 1$. Then $\eta_0'$ is nuclearly absorbing.
\end{prop}

\begin{proof}
The fact that $\eta_0'$ is unitizably full follows verbatim from {\cite[Theorem 4.4]{chris1}}. Since $J_0$ is a separable admissible kernel, $J_0$ is stable and has the corona factorization property {\cite[Proposition 3.3.(3)]{chris2}}. Combining this with the fact that $\eta_0'$ is unitizably full, we get that $\eta_0'$ is nuclearly absorbing by Theorem $2.6$ in \cite{gabe}.
\end{proof}

Finally, the proof of our main result is a combination of the previous lemmas, so let us assume the same notation as in the pullback extensions built in \eqref{septargetext} and \eqref{sepsplitext}. Having our previous results at hand, the non-unital case follows verbatim as in {\cite[Theorem 1.4]{chris1}}. For the case when $\tau$ is a tracial state, instead of following the proof of {\cite[Theorem 1.4]{chris1}} which is using classification of normal $^*$-homomorphisms into $\II_1$-factors, we are going to use Lemma \ref{tracecontraction}.

\begin{proof}[Proof of Theorem \ref{no1}]
Suppose first that $A$ is non-unital or $A$ is unital and $\phi_0(1)\neq 1$. By combining Proposition \ref{class0} and Proposition \ref{absorbing}, we get that $\eta_0'$ has a weakly nuclear $^*$-homomorphic splitting. Then, by making use of exactness of $A$ and faithfulness of $\tau$, one can use the proof of Theorem $1.4$ in \cite{chris1} to get a nuclear $^*$-homomorphism $\psi_0:A\rightarrow\mathcal{Q}_{\omega}$ such that $\pi\circ\psi_0=\phi$. But since $\psi_0$ is nuclear, the Choi-Effros lifting theorem implies that $\psi_0$ has a cpc lift $A\rightarrow \ell^{\infty}(\mathcal{Q})$. Moreover, by choice of $\phi$, we have $\tr^{\omega}\circ\phi=\tau$, so $$\tau=\tr_{\omega}\circ\pi\circ\phi=\tr_{\omega}\circ\psi_0.$$ Thus, by Proposition \ref{keyprop}, $\tau$ is quasidiagonal. 

Finally, if $A$ is unital and $\phi_0$ is unital, then $\tau$ is a tracial state. By Lemma \ref{tracecontraction}, $\frac{1}{2}\tau$ is amenable, so one can consider the same problem with $\frac{1}{2}\tau$ and $\frac{1}{2}\tau\circ\sigma$, and since $\frac{1}{2}\tau$ is not a state, the proof above will give that $\frac{1}{2}\tau$ is quasidiagonal. Hence, $\tau$ is quasidiagonal by Lemma \ref{tracecontraction}.
\end{proof}

Theorem \ref{no2} follows easily by unravelling the definitions and using a simple trick to pass quasidiagonality from a faithful trace to all amenable traces on $A$.

\begin{proof}[Proof of Theorem \ref{no2}]
Let $\phi:A\rightarrow B$ and $\psi:B\rightarrow A$ be two $^*$-homomorphisms such that $\psi\circ\phi$ is homotopic to the identity map on $A$. 

Now let $\tau_0$ be a faithful amenable trace on $A$. Then, $\tau_0\circ\psi$ is amenable on $B$, so quasidiagonal by assumption. Thus, precomposing with $\phi$ gives a quasidiagonal trace on $A$. Therefore, by Theorem \ref{no1}, $\tau_0$ must be quasidiagonal on $A$. If $\tau$ is any amenable trace on $A$, for any $n\in\mathbb{N}$ the convex combination $\tau_n:=\frac{n-1}{n}\tau+\frac{1}{n}\tau_0$ is a faithful amenable trace on $A$, so quasidiagonal by the same argument. 

Then, the set of quasidiagonal traces is weak$^*$-closed {\cite[Proposition 3.5.1]{qdtraces}} and $\tau_n$ converges weak$^*$ to $\tau$, so $\tau$ must be quasidiagonal.
\end{proof}

\begin{question}
Is the property that all amenable traces are quasidiagonal invariant under homotopy for arbitrary separable $C^*$-algebras?
\end{question}

\bibliography{traces}
\bibliographystyle{abbrv}
\end{document}